\documentclass[10pt]{article}
\font\smallit=cmti10
\font\smalltt=cmtt10

\usepackage{amssymb,latexsym,amsmath,epsfig,amsthm} 

\makeatletter

\renewcommand\section{\@startsection {section}{1}{\z@}
{-30pt \@plus -1ex \@minus -.2ex}
{2.3ex \@plus.2ex}
{\normalfont\normalsize\bfseries}}

\renewcommand\subsection{\@startsection{subsection}{2}{\z@}
{-3.25ex\@plus -1ex \@minus -.2ex}
{1.5ex \@plus .2ex}
{\normalfont\normalsize\bfseries}}

\renewcommand{\@seccntformat}[1]{\csname the#1\endcsname. }

\makeatother

\newtheorem{theorem}{Theorem}
\newtheorem{lemma}[theorem]{Lemma}

\newtheorem*{EPNT}{The pentagonal number theorem}
\newtheorem*{PGF}{The generating function for $p(n)$}

\theoremstyle{definition}

\begin{document}
\begin{center}
\uppercase{\bf Partition Recurrences}
\vskip 20pt
{\bf Yuriy Choliy}\\
{\smallit East Hanover, New Jersey, USA}\\
{\tt yurkocholiy@hotmail.com}\\ 
\vskip 10pt
{\bf Louis W. Kolitsch}\\
{\smallit Department of Mathematics and Statistics, University of Tennessee--Martin, Martin,
Tennessee, USA}\\
{\tt lkolitsc@utm.edu}\\ 
\vskip 10pt
{\bf Andrew V. Sills}\\
{\smallit Department of Mathematical Sciences, Georgia Southern University, Statesboro,
Georgia, USA}\\
{\tt asills@georgiasouthern.edu}\\ 
\end{center}
\vskip 30pt
\centerline{\smallit Received: November 13, 2016, Revised: May 15, 2017, Accepted: , Published: } 
\vskip 30pt

\centerline{\bf Abstract}
\noindent
We present some Euler-type recurrences for the partition function $p(n)$.

\pagestyle{myheadings} 
\markright{\smalltt INTEGERS: ?? (201?)\hfill} 
\thispagestyle{empty} 
\baselineskip=12.875pt 
\vskip 30pt

\section{Introduction and Statement of Main Results}
\subsection{Euler's recurrence}
A \emph{partition} $\lambda$ of an integer $n$ is a finite, weakly decreasing 
sequence $(\lambda_1, \lambda_2, \dots, \lambda_l)$ of positive
integers (called the \emph{parts} of $\lambda$)
that sum to $n$.  Thus the seven partitions of $5$ are
\[ (5)\quad (4,1) \quad (3,2) \quad (3,1,1) \quad (2,2,1) \quad (2,1,1,1)\quad (1,1,1,1,1). \]
Let $p(n)$ denote the number of partitions of the integer $n$.
Euler proved the following recurrence for $p(n)$~\cite[p. 12, Cor. 1.8]{a76}:
\begin{multline} \label{EulerRec}
  p(n) - p(n-1) - p(n-2) + p(n-5) + p(n-7) - p(n-12) - p(n-15)  \\
  + \cdots + (-1)^j p(n- j(3j-1)/2) + (-1)^j p(n-j(3j+1)/2)+ \cdots \\
  = \left\{ \begin{array}{ll} 1 & \mbox{if $n=0$} \\
                0  & \mbox{otherwise} 
  \end{array}. \right. 
\end{multline}
The numbers $0,1,2,5,7,12,\dots$ appearing in~\eqref{EulerRec} are the
generalized pentagonal numbers,
and since $p(n)=0$ whenever $n<0$,
the left side of~\eqref{EulerRec} effectively terminates after roughly
$\sqrt{8n/3}$ terms.

As is well known, Euler's recurrence~\eqref{EulerRec} is a direct consequence of
two fundamental results of Euler
on partitions~\cite[p. 4, Eq. (1.2.3); p. 11, Cor. 1.7]{a76}, as follows.
Note that here and throughout, $x$ represents a formal variable.

\begin{PGF}
Let
$ P(x):= \sum_{n\geq 0} p(n) x^n. $ Then
\[ P(x) = \prod_{k\geq 1} \frac{1}{1-x^k}. \]
\end{PGF}

\begin{EPNT}
\[  \frac{1}{P(x)} = \sum_{k\in\mathbb{Z}} (-1)^k x^{k(3k-1)/2}. \]
\end{EPNT}

Therefore, 
\begin{multline} \label{2} 1 = P(x) \cdot \frac{1}{P(x)} = 
\left( \sum_{j\geq 0} p(j) x^j \right)\left( \sum_{k\in\mathbb{Z}} (-1)^k x^{k(3k-1)/2} \right)\\
 \sum_{n\geq 0} \left( \sum_{j\in\mathbb{Z}} (-1)^j p\Big(n-j(3j-1)/2 \Big) \right) x^n, \end{multline}
and~\eqref{EulerRec} follows from extracting the coefficient of $x^n$
from the extremes of~\eqref{2}.

The role played by the pentagonal numbers in~\eqref{EulerRec} is
played by the
 triangular numbers in the following recurrence due to 
 John Ewell~\cite[p. 125, Theorem 2]{e73}:
  For all integers $n$,
\begin{multline}  \label{triangular}
 p(n) - p(n-1) - p(n-3) + p(n-6) + p(n-10) - p(n-15) - p(n-21) \\+ \cdots + 
   (-1)^j p(n- (2j^2-j) ) + (-1)^j p(n- (2j^2+j))+ \cdots \\ = 
  \left\{ \begin{array}{ll} 0 & \mbox{if $n$ is odd} \\
                q(n/2)  & \mbox{if $n$ is even} 
  \end{array}, \right. \end{multline}
  where $q(n)$ denotes the number of partitions of $n$ into distinct parts.
  Ewell provides additional recurrences for $p(n)$ in~\cite{e73} and in~\cite{e04}.
\subsection{Some notation}

We will employ the notations of Ramanujan~\cite[p. 6, Eq. (1.2.3)]{b06},
 \[  \psi(x):= \sum_{n\geq 0} x^{n(n+1)/2} ,\]
and~\cite[p. 6, Eq. (1.2.2)]{b06}:
  \[ \varphi(x) = 1 + 2\sum_{j\geq 1} x^{j^2}. \]

\subsection{Statement of results}

In this paper, we present several partition recurrences, in the spirit of Euler
and Ewell.

 A partition recurrence where the squares and their doubles play the same role
 as the generalized pentagonal numbers in~\eqref{EulerRec} is as follows.
\begin{theorem} \label{squares}
For all integers $n$,
\begin{multline}
p(n) - p(n-1) - p(n-2) + p(n-4) + p(n-8) - p(n-9) - p(n-18) + p(n-16) + \cdots \\
 \cdots + (-1)^j p(n-j^2) + (-1)^j p(n-2j^2) + \cdots \\ =
   \left\{ \begin{array}{ll} 0 & \mbox{if $n$ is odd} \\
                qq(n)  & \mbox{if $n$ is even} 
  \end{array}, \right. \label{SquareRec} 
\end{multline}
where $qq(n)$ denotes the number of partitions of $n$ into distinct, odd parts.
\end{theorem}

The following theorem is a variation of Theorem~\ref{squares}.
\begin{theorem} \label{squaresvar}
For all integers $n$,
\begin{multline}
p(n) - 2p(n-1) + 2p(n-4) - 2p(n-9) + 2p(n-16) + \cdots +(-1)^j 2p(n-j^2) + \cdots \\ =
     (-1)^n qq(n)   \label{SquareVar} 
\end{multline}
where $qq(n)$ denotes the number of partitions of $n$ into distinct, odd parts.
\end{theorem}

Continuing with the polygonal numbers, we have the following theorem involving the generalized heptagonal numbers.
\begin{theorem} \label{heptagonal}
If $n$ is even, then
\[ p(n) - p(n-1) - p(n-4) + p(n-7) + p(n-13) - p(n-18) - p(n-27) + \cdots \]
\[ \cdots + (-1)^j p(n-j(5j-3)/2) + (-1)^j p(n-j(5j+3)/2) + \cdots \] 
\[ = \sum_{i\in \mathbb{Z}} \left(p\left(\frac{n-r^e_i }{2 }  \right) - p\left(\frac{n-s^e_i }{2 }  \right)\right). \]
If $n$ is odd, then
\[ p(n) - p(n-1) - p(n-4) + p(n-7) + p(n-13) - p(n-18) - p(n-27) + \cdots \]
\[ \cdots + (-1)^j p(n-j(5j-3)/2) + (-1)^j p(n-j(5j+3)/2) + \cdots \] 
\[ = \sum_{i\in \mathbb{Z}} \left(-p\left(\frac{n-r^o_i }{2 }  \right) + p\left(\frac{n-s^o_i }{2 } \right)\right) \] 
where $r^e_i$ and $r^o_i$ are the even and odd terms, respectively, in the sequence $15k^2-4k$ for $k$ an integer and $s^e_i$ and $s^o_i$ are the even and odd terms, respectively, in the sequence $15k^2-14k+3$ for $k$ an integer.  Specifically for $i \in \mathbb{Z}$, \[r^e_i = 15(2i)^2 -4(2i) = 60i^2-8i, \] \[r^o_i = 15(2i+1)^2-4(2i+1) = 60i^2+52i+11,\] \[s^e_i = 15(2i+1)^2-14(2i+1)+3 = 60i^2+32i+4,\] \[ s^o_i = 15(2i)^2 - 14(2i) + 3 = 60i^2-28i + 3.\]
\end{theorem}

The following recurrence involves the generalized octagonal numbers.
\begin{theorem} \label{octagonal} If $n$ is even, then
\[ p(n) - p(n-1) - p(n-5) + p(n-8) + p(n-16) - p(n-21) - p(n-33) + \cdots \]
\[ \cdots + (-1)^j p(n-j(3j-2)) + (-1)^j p(n-j(3j+2)) + \cdots \]  
\[= \sum_{i\geq 1} p\left(\frac{n-3t^e_i }{2 }\right) \]
where $t^e_i$ is the $i$th even triangular number ($t^e_1 = 0$, $t^e_2 = 6$, $t^e_3 = 10$,
etc.), i.e.,
\[ t^e_i = \frac{(2i-1)(2i-1+(-1)^i)}{2}. \]
If $n$ is odd, then 
\[ p(n) - p(n-1) - p(n-5) + p(n-8) + p(n-16) - p(n-21) - p(n-33) + \cdots \]
\[ \cdots + (-1)^j p(n-j(3j-2)) + (-1)^j p(n-j(3j+2)) + \cdots  \]
\[ = \sum_{i\geq 1} p\left(\frac{n-3t^o_i }{2 }\right) \]
where $t^o_i$ is the $i$th odd triangular number, i.e.,
 \[ t^o_i = \frac{(2i-1)(2i-1-(-1)^i)}{2}\].   
\end{theorem}

  A two-color partition of $n$ is similar to a
partition of $n$, but where a given positive integer can
appear as a part in either of two distinguished kinds.  For instance, if our two ``colors"
are represented by ordinary and boldface type, the five two-color partitions of $2$ are
\[ (2) \quad (\mathbf{2}) \quad (1,1) \quad (1,\mathbf{1}) \quad  (\mathbf{1},\mathbf{1}).\]
Let $p_2(n)$ denote the number of two-color partitions of $n$.  

Notice that 
since $p_2(n)$
is the coefficient of $x^n$ in the series expansion of $[P(x)]^2$, we
immediately have
\[ p_2(n) = \sum_{i=0}^n p(i) p(n-i), \] and thus Theorem~\ref{p_2} below could be
rephrased entirely in terms of the partition function $p(n)$, if we so choose.
\begin{theorem} \label{p_2} If $n$ is even, then
\[  p(n) = \sum_{i\geq 1} p_2\left(\frac{n-t^e_i }{2 }  \right),\]
where $t^e_i$ is the $i$th even triangular number. 
If $n$ is odd, then
\[  p(n) = \sum_{i\geq 1} p_2\left(\frac{n-t^o_i }{2 }  \right),\]
where $t^o_i$ is the $i$th odd triangular number.
\end{theorem}

  Overpartitions were introduced by S. Corteel and J. Lovejoy in~\cite{cl04}, and
have since accumulated a rather vast literature.  
  An \emph{overpartition} of $n$ is a two-color partition of $n$ where a given 
positive integer in the second color may appear at most once as a part.  
(Parts in the first color may be repeated any number of times.) Traditionally,
parts in the first color are denoted as ordinary numerals and parts in the 
second color are denoted as overlined numerals.  Thus the eight overpartitions
of $3$ are as follows:

\[ (3) \quad (\overline{3}) \quad (2,1) \quad (\overline{2},1) \quad
(2,\overline{1}) \quad (\overline{2}, \overline{1}) \quad (1,1,1) \quad
(\overline{1}, 1, 1) .\]

  Just as is the case with ordinary partitions, restrictions can be placed on the parts in overpartitions.  Overpartitions and a restricted class of overpartitions make an unexpected appearance in the following result.
  
\begin{theorem} 
\label{vThm}
 Let $\overline{p}(n)$ denote the number of unrestricted overpartitions of $n$ and let $\overline{p}_{r}(n)$ denote the number of overpartitions of $n$ where the overlined parts are either even and appear in one color or are congruent to $1$ or $7$ modulo $8$ and can appear in two colors.  Let $v(n)$ denote the sequence determined by pairing the left side of~\eqref{SquareRec} with the right side of~\eqref{EulerRec}, namely
\begin{multline}
v(n) - v(n-1) - v(n-2) + v(n-4) + v(n-8) - v(n-9) - v(n-18) + v(n-16) + \cdots \\
 \cdots + (-1)^j v(n-j^2) + (-1)^j v(n-2j^2) + \cdots \\ =
   \left\{ \begin{array}{ll} 1 & \mbox{if $n=0$,} \\
               0  & \mbox{otherwise} 
  \end{array}. \right. \label{vRec} 
\end{multline}
Then, for $n$ even, $v(n) = \overline{p}(n/2)$  and for $n$ odd, $v(n) = \overline{p}_{r}((n-1)/2)$.
\end{theorem}

\vskip 6mm
In~\cite{m21} (cf.  \cite[p. 234, Thm. 14.1]{a76}), 
P. A. MacMahon presents a recurrence that permits the 
determination of the parity of $p(n)$ more quickly than using~\eqref{EulerRec} directly,
namely
\begin{equation} \label{MacMahonRec}
p(n) \equiv \sum_{t\in T} p\left( \frac{n-t}{4} \right) \pmod{2},
\end{equation}
where $T = T(n)$ is the set of integers such that $t$ is triangular, $0\leq t\leq n$,
and $t\equiv n \pmod{4}$.   
MacMahon's recurrence~\eqref{MacMahonRec} follows from the following sequence of 
observations:  It follows that
\[ \sum_{t\in T} p\left( \frac{n-t}{4} \right)  = qq(n) \] by extracting the coefficient of $x^n$ in the
extremes of
\begin{align*}
\sum_{n\geq 0} qq(n) x^n &= \prod_{k\geq 1} (1+x^{2k-1}) \\
&=  \prod_{k\geq 1} \frac{(1+x^{4k-3})(1+x^{4k-1})(1-x^{4k}) }{1-x^{4k} } \\
& = \left( \sum_{j\geq 0} p(j) x^{4j} \right) \psi(x) \\
& = \sum_{n\geq 0} \sum_{t\in T}  p\left( \frac{n-t}{4} \right) x^n.
\end{align*}
Next, we observe that $qq(n) \equiv p(n)\pmod{2}$ for all integers $n$, because

\begin{align*}
\sum_{n\geq 0} p(n) x^n & = \prod_{k\geq 1} \frac{1}{1-x^k}\\
&\equiv \prod_{k\geq 1} \frac{1}{1+x^{k}} \pmod{2} \\
& = \prod_{k\geq 1} (1-x^{2k-1}) \\
&\equiv\prod_{k\geq 1} (1+x^{2k-1}) \pmod{2} \\
& = \sum_{n\geq 0} qq(n) x^n.
\end{align*}

Here we present a different Euler-type recurrence that quickly determines the parity
of $p(n)$.

\begin{theorem} \label{ParityRec}
For all integers $n$,
\begin{equation}  
\sum_{j\geq 0} p(n- 4\pi_j) \equiv
    \left\{ \begin{array}{ll} 1 \pmod{2} & \mbox{if $n$ is a triangular number,} \\
                0\pmod{2} & \mbox{otherwise} 
  \end{array} \right.
\end{equation}  
where $\pi_j$ is the $j$th generalized pentagonal number, i.e.,
\[ \pi_j = \frac 38 j^2 + \frac{3-(-1)^j}{8} j + \frac{1-(-1)^j}{16}. \]
\end{theorem}

Before providing proofs of our results in the next section, we acknowledge 
analogous results by other authors.
In addition to the work of Ewell~\cite{e73,e04} previously mentioned,
  K. Ono, N. Robbins, and B. Wilson~\cite{onw96} 
and Robbins~\cite{r02} have proved recurrences
of this type for $qq(n)$.

\section{Proofs of Theorems~\ref{squares}--\ref{ParityRec}}

\subsection{Proof of Theorem~\ref{squares}}

Since 
   \[  \sum_{n\geq 0} qq(n) x^n = \prod_{k\geq 1} (1+x^{2k+1}), \]
and for any power series 
  \[ f(x) = a_0 + a_1 x + a_2 x^2 + a_3 x^3 + \cdots, \] it is the case that
  
 \begin{equation} \label{ExtractEvens} 
 \frac{f(x) + f(-x)}{2} = a_0 + a_2 x^2 + a_4 x^4 + a_6 x^6 + \cdots, \end{equation}
we have that
\begin{align*}
\sum_{n\geq 0} \frac{1+(-1)^n}{2} qq(n) x^n &=
\sum_{k\geq 0} qq(2k) x^{2k} \\
&= \frac{1}{2} \left(  \prod_{k\geq 1} (1+x^{2k-1})
+ \prod_{k\geq 1} (1-x^{2k-1}) \right) \\
& = \frac 12\prod_{k\geq 1} \frac{1-x^{4k-2}}{1-x^{2k-1}} 
+ \frac 12\prod_{k\geq 1} \frac{(1-x^{2k-1})(1-x^k)}{1-x^{k}}\\
&= \frac{ \varphi(-x^2) + \varphi(-x)}{2 \prod_{k\geq 1} (1-x^k)}\\
& = \left( \sum_{k\geq 0} p(k) x^k \right) \left( 1 + \sum_{j\geq 1} (-1)^j x^{2j^2} +
\sum_{j\geq 1} (-1)^j x^{j^2} \right)\\
&= \sum_{n\geq 0} \left( p(n) + \sum_{j\geq 1} (-1)^j \left[ p(n-2j^2) + p(n-j^2) \right]  \right) x^n.
\end{align*}
Note that we are using \[ \varphi(x) = \prod_{k\geq1}(1+x^{2k-1})^2(1-x^{2k}) \] (see, e.g.,~\cite[p. 11, Eq. (1.3.13)]{b06}).
The result then follows from extracting the coefficient of $x^n$ in the extremes. \qed

\subsection{Proof of Theorem~\ref{squaresvar}}

Since 
\begin{align*}
\sum_{n\geq 0} (-1)^n qq(n) x^n 
&= \prod_{k\geq 1} (1-x^{2k-1}) \\
&= \prod_{k\geq1} \frac{(1-x^{2k-1})(1-x^{k})}{(1-x^k)}\\
&=\frac{\varphi(-x)}{\prod_{k\geq 1}{1-x^k}}\\
& = \left( \sum_{k\geq 0} p(k) x^k \right) \left( 1 + 2\sum_{j\geq 1} (-1)^j x^{j^2} \right)\\
\end{align*}
The result then follows from extracting the coefficient of $x^n$ in the extremes. \qed

\subsection{Proof of Theorem~\ref{heptagonal}}

By Jacobi's triple product identity~\cite[p. 11, Eq. (1.3.14)]{b06},
\[  \sum_{j=-\infty}^\infty (-1)^j x^{j(5j-3)/2} = \prod_{k\geq 1} (1-x^{5k})(1-x^{5k-4})(1-x^{5k-1}). \]
Thus, 
\[ \left(\sum_{k\geq 0} p(k) x^k \right) \left( 1 + \sum_{j\geq 1} (-1)^j x^{j(5j-3)/2} + \sum_{j\geq 1} (-1)^j x^{j(5j+3)/2} \right) \] \\
\[=\prod_{k\geq 1} \frac {1} {(1-x^{5k-2})(1-x^{5k-3})} \] \\
\[=\prod_{k\geq1}\frac {(1-x^{10k})(1-x^{20k-16})(1-x^{20k-4})(1+x^{10k-7})(1+x^{10k-3})}{1-x^{2k}} \] \\
\[=\left( \sum_{k\geq 0} p(k) x^{2k} \right) \left( \sum_{j=-\infty}^{\infty} (-1)^j \left(x^{15j^2-4j}+x^{15j^2+14j+3}\right)\right),\] \\
where the fact that \[ \prod_{k\geq1} {(1-x^{10k})(1-x^{20k-16})(1-x^{20k-4})(1+x^{10k-7})(1+x^{10k-3})} \] \[= \sum_{j=-\infty}^{\infty}  (-1)^j \left(x^{15j^2-4j}+x^{15j^2+14j+3}\right) \] follows from the quintuple product identity~\cite[p. 18, Eq. (1.3.52)]{b06}.
The result then follows from extracting the coefficient of $x^n$ in the extremes.   \qed

\subsection{Proof of Theorem~\ref{octagonal}}

By Jacobi's triple product identity,
\[  \sum_{j=-\infty}^{\infty}  (-1)^j x^{3j^2-2j} = \prod_{k\geq 1} 
(1-x^{6k})(1-x^{6k-5})(1-x^{6k-1}). \]
Thus, 
\begin{align*}
\left(\sum_{k\geq 0} p(k) x^k \right) \left( \sum_{j=-\infty}^{\infty}  
(-1)^j x^{3j^2-2j} \right) &= 
\prod_{k\geq 1} \frac {1} {(1-x^{6k-4})(1-x^{6k-3})(1-x^{6k-2})} \\
&=\prod_{k\geq1}\frac {(1-x^{12k})(1+x^{12k-9})(1+x^{12k-3})}{1-x^{2k}}  \\
&=\left(\prod_{k\geq1}\frac {1} {1-x^{2k}} \right) \psi(x^3) \\
&=\left( \sum_{k\geq 0} p(k) x^{2k} \right) \left(\sum_{j\geq 0} x^{3t_j}\right) 
\end{align*}
The result then follows from extracting the coefficient of $x^n$ in the extremes.  \qed

\subsection{Proof of Theorem~\ref{p_2} }
First we introduce a lemma.

\begin{lemma} \label{lem}
For functions $P$ and $\psi$ defined above, the following holds:
  \begin{equation} \label{L}
  P(x) \pm P(-x) = P(x^2)^2 \left( \psi(x) \pm \psi(-x) \right). \end{equation}
\end{lemma}
\begin{proof}
  By a theorem of Euler, 
    \begin{equation}
      P(x) = \prod_{k\geq 1} \frac{1}{1-x^k}.
    \end{equation}
  By Jacobi's triple product identity followed by some algebraic manipulation,
    \begin{equation}
       \psi(x) = \prod_{k\geq 1} \frac{1-x^{2k}}{1-x^{2k-1}}.
    \end{equation}  
Thus 
\begin{align*}
P(x^2)^2 \left( \psi(x) \pm \psi(-x) \right)  &= 
\prod_{k\geq 1}\frac{1}{(1-x^{2k})^2}
 \left(  \prod_{k\geq 1} \frac{1-x^{2k}}{1-x^{2k-1}} \pm  
  \prod_{k\geq 1} \frac{1-x^{2k}}{1+x^{2k-1}} \right) \\
 & = \prod_{k\geq 1}\frac{1}{(1-x^{2k})}
 \left(  \prod_{k\geq 1} \frac{1}{1-x^{2k-1}} \pm  
  \prod_{k\geq 1} \frac{1}{1+x^{2k-1}} \right) \\
  & = \prod_{k\geq 1} \frac{1}{(1-x^{2k})(1-x^{2k-1})} \pm  
  \prod_{k\geq 1} \frac{1}{(1-x^{2k})(1+x^{2k-1})} \\
  &=  \prod_{k\geq 1} \frac{1}{1-x^k} \pm  \prod_{k\geq 1} \frac{1}{1-(-x)^k}\\
  &=P(x) \pm P(-x).
\end{align*}
\end{proof}

Now we may prove Theorem~\ref{p_2}.
\begin{proof}[Proof of Theorem~\ref{p_2}]
By Eq.~\eqref{ExtractEvens},
\[ \frac{P(x) + P(-x)}{2} = p(0) + p(2) x^2  + p(4) x^4 + p(6) x^6 + \cdots . \]
Also, by Lemma~\ref{lem},
\begin{multline*}
P(x^2)^2 \left( \frac{ \psi(x) + \psi(-x)}{2} \right)  \\
= \Big(  p_2(0) + p_2(1) x^2 + p_2(2) x^4 + \cdots \Big) ( x^0+ x^6 + x^{10} + x^{28} + x^{36}+\cdots )
\end{multline*}
The result for even $n$ thus follows from extracting the coefficient of $x^n$ from both
sides of
\[ \frac{P(x) + P(-x)}{2} = P(x^2)^2 \left( \frac{ \psi(x) + \psi(-x)}{2} \right) . \]

By reasoning similar to that of~\eqref{ExtractEvens},
\[ \frac{P(x) - P(-x)}{2} = p(1)x + p(3) x^3  + p(5) x^5 + p(7) x^7 + \cdots . \]
Also,
\begin{multline*}
P(x^2)^2 \left( \frac{ \psi(x) - \psi(-x)}{2} \right)  \\
= \Big(  p_2(0) + p_2(1) x^2 + p_2(2) x^4 + \cdots \Big) ( x+ x^3 + x^{15} + x^{21} +\cdots ).
\end{multline*}
The result for odd $n$ thus follows from extracting the coefficient of $x^n$ from both
sides of
\[ \frac{P(x) - P(-x)}{2} = P(x^2)^2 \left( \frac{ \psi(x) - \psi(-x)}{2} \right) . \]
\end{proof}

\subsection{Proof of Theorem~\ref{vRec}}

First we state a lemma.
\begin{lemma} \label{PentLemma}
Let \[ F(x) = \frac {1} {P(x)} = \prod_{k\geq1}(1-x^k). \] Then  
  \[ \frac {F(x) - F(-x)} {2} = -x\prod_{k\geq 1} (1-x^{16k-6})(1-x^{16k-10})(1-x^{16k})(1-x^{32k-4})(1-x^{32k-28}) \] 
  \center{and}
  \[ \frac {F(x) + F(-x)} {2} = \prod_{k\geq 1} (1-x^{16k-2})(1-x^{16k-14})(1-x^{16k})(1-x^{32k-12})(1-x^{32k-20}). \]
\end{lemma}
\begin{proof}
The terms in the series 
\[ \frac {F(x) - F(-x)} {2} \] 
are the odd power terms in \[ \sum_{k\in\mathbb{Z}} (-1)^k x^{k(3k-1)/2} \] and thus 
\begin{align*}
\frac {F(x) - F(-x)} {2} &= \sum_{j=-\infty}^{\infty}(x^{24j^2+26j+7}-x^{24j^2-10j+1}) \\
&=-x\prod_{k\geq 1}(1-x^{16k-6})(1-x^{16k-10})(1-x^{16k})(1-x^{32k-4})(1-x^{32k-28}), 
\end{align*}
where the last equality follows from the quintuple product identity.\
The terms in the series 
\[ \frac {F(x) + F(-x)} {2} \] 
are the even power terms in \[ \sum_{k\in\mathbb{Z}} (-1)^k x^{k(3k-1)/2}, \] and thus 
\begin{align*}
\frac {F(x) + F(-x)} {2} &= \sum_{j=-\infty}^{\infty}(x^{24j^2+2j}-x^{24j^2+14j+2}) \\
&=\prod_{k\geq 1}(1-x^{16k-2})(1-x^{16k-14})(1-x^{16k})(1-x^{32k-12})(1-x^{32k-20}) 
\end{align*}
where the last equality follows from the quintuple product identity.
\end{proof}

Now we proceed to a proof of Theorem~\ref{vRec}.
\begin{proof}[Proof of Theorem~\ref{vRec}]
Let $$V(x) := \sum_{n\geq 0} v(n) x^n.$$
By the definition of $v(n)$, and from information gathered during the proof of 
Theorem~\ref{squares}, it must be the case that
 \[ \frac{1}{V(x)} = \frac{ \varphi(-x) + \varphi(-x^2)}{2}. \]
Thus, we have immediately that
 \[ V(x) =  \frac{2}{ \varphi(-x) + \varphi(-x^2)}. \]
Using the product form of $\varphi(x)$, we can express 
\begin{align*}
 V(x) &= \frac {2} {\left(\prod_{k\geq 1} (1-x^{2k-1}) \right)
 \left(F(x) + F(-x) \right)} \\
 &= \frac {2F(x^2)} {F(x) \left(F(x) + F(-x) \right)}.
 \end{align*} 
For the first part of the theorem, we wish to extract the even powers of $V(x)$, and so we find
\begin{align*}
\frac{V(x) + V(-x)}{2} &=  \frac {F(x^2)} {F(x)F(-x)} \\
 &=\prod_{k\geq 1}\frac {1+x^{2k}} {1-x^{2k} }.
 \end{align*}
Recalling that 
\[ \sum_{n\geq 0} \overline{p}(n) x^n = \prod_{k\geq 1} \frac {1+x^{k}} {1-x^{k}}, \] and that we
have just shown that
\[ \sum_{j\geq 0} v(2j) x^j=
 \frac{V(x^{1/2}) + V(-x^{1/2})}{2} = 
 \prod_{k\geq 1}\frac {1+x^{k}} {1-x^{k}}, \] we conclude
that $v(n) = \overline{p}(n/2)$ when $n$ is even.

For the second part of the theorem, we wish to extract the odd powers of $V(x)$, and so we find

\[ \frac{V(x) - V(-x)}{2} =  \frac {F(x^2)\left( F(-x)-F(x) \right)} {F(x)F(-x)\left( F(x)+F(-x) \right)} \]
\[ = x\prod_{k\geq 1} \frac { (1+x^{2k})(1-x^{16k-6})(1-x^{16k-10})(1-x^{16k})(1-x^{32k-4})(1-x^{32k-28})} { (1-x^{2k})(1-x^{16k-2})(1-x^{16k-14})(1-x^{16k})(1-x^{32k-12})(1-x^{32k-20})} \]
\begin{align*}
&=x\prod_{k\geq 1}\frac { (1+x^{2k})(1+x^{16k-2})(1+x^{16k-14})} 
{ (1-x^{2k})(1+x^{16k-6})(1+x^{16k-10})} \\
&=x\prod_{k\geq 1}\frac { (1+x^{4k})(1+x^{16k-2})^2(1+x^{16k-14})^2} {1-x^{2k}} .
\end{align*}

Noting that the generating function for our restricted overpartitions is   
\[ \sum_{n\geq 0} \overline{p}_{r}(n) x^n = 
\prod_{k\geq 1}\frac { (1+x^{2k})(1+x^{8k-1})^2(1+x^{8k-7})^2} 
{ 1-x^{k}}, \] and that we
have just shown that
\begin{align*}
\sum_{j\geq 0} v(2j+1) x^j &= \frac{V(x^{1/2}) - V(-x^{1/2})}{2x^{1/2}} \\
&= \prod_{k\geq 1}\frac {(1+x^{2k})(1+x^{8k-1})^2(1+x^{8k-7})^2} { 1-x^{k}}, 
\end{align*}
we conclude
that $v(n) = \overline{p}_{r}((n-1)/2)$ when $n$ is odd.

\end{proof}

\subsection{Proof of Theorem~\ref{ParityRec}}

For two power series $f(x) = \sum_{n\geq 0} a_n x^n$ and 
$g(x) = \sum_{n\geq 0} b_n x^n$, we write
\[ f(x) \equiv g(x) \pmod{m}, \] if
$a_n \equiv b_n \pmod{m}$ for all $n\geq 0$.

Let $\chi(S) = 1$ if the statement $S$ is true and $0$ if $S$ is false.

\begin{align*}
\sum_{n\geq 0} \chi(\mbox{$n$ is triangular}) x^n &=
\sum_{k\geq 0} x^{k(k+1)/2} \\
&= \prod_{k\geq 1} \frac{1- x^{2k}}{1 - x^{2k-1}} \\
& \equiv \prod_{k\geq 1} \frac{1+ x^{2k}}{1 - x^{2k-1}} \pmod{2} \\
& = \prod_{k\geq 1} (1+x^k)(1+ x^{2k}) \quad\mbox{by~\cite[p. 4, Eq. (1.1.9)]{b06}}\\
& = \prod_{k\geq 1} \frac{1}{1-x^k} \cdot (1-x^k)(1+x^k)(1+x^{2k}) \\
& = \prod_{k\geq 1} \frac{1}{1-x^k} \cdot (1-x^{2k})(1+x^{2k}) \\
& = \prod_{k\geq 1} \frac{1}{1-x^k} \cdot (1-x^{4k}) \\
&= \left( \sum_{k\geq 0} p(k) x^k \right) \left( \sum_{j\in\mathbb{Z}} (-1)^j x^{6j^2-2j} \right)\\
&= \sum_{n\geq 0} \left( \sum_{j\geq 0} p(n-4\pi_j) \right) x^n.
\end{align*}
The result follows from extracting the coefficient of $x^n$ in the extremes. \qed

\section{Conclusion}
We make no claims that our results improve upon Euler's recurrence in terms of their
computational efficiency.  In this regard (as is so often the case), Euler 
surely remains king.
Our main purpose here is to place Euler's recurrence in a larger context, and to present
some aesthetically pleasing results analogous to Euler's.

\section*{Acknowledgments}
We thank Dennis Eichhorn and James Sellers for pointing us to several
relevant references in the literature, and the anonymous referee who read our
manuscript carefully and pointed out a number of typographical errors.  
We thank Bruce Landman for 
organizing the INTEGERS 2016 conference, which facilitated the 
first and third authors becoming aware of the second author's 
earlier unpublished work on
this topic.  The result is this three-author joint paper.

\end{document}